\documentclass[letterpaper, 10 pt, conference]{ieeeconf}  

\IEEEoverridecommandlockouts                              
\usepackage{graphics} 
\usepackage{epsfig} 
\usepackage{times} 
\usepackage{amsmath} 
\usepackage{amssymb}  
\usepackage{enumerate}
\usepackage{accents}
\usepackage{siunitx}
\usepackage{tikz}
\usepackage[mathscr]{euscript}
\usepackage{caption}
\usepackage{mathtools}
\usepackage{subcaption}
\usetikzlibrary{calc, patterns, decorations.pathmorphing, decorations.markings, shapes.geometric, arrows, positioning, angles}
\usepackage{textcomp}
\usepackage{hyperref}
\usepackage{lipsum}
\newcommand{\ubar}[1]{\underaccent{\bar}{#1}}

\newtheorem{prop}{Proposition}

\newtheorem{definition}{Definition}

\DeclareMathOperator*{\argmin}{arg\,min}
\title{\LARGE \bf
	Lyapunov Safety Observers for a Restricted Class of Nonlinear Systems
} 
\title{\LARGE \bf
Lyapunov Safety Observers: Enforcing a Safe State-Space
} 
\title{\LARGE \bf
Safety Enforcement in Constrained Nonlinear Systems via Min-Quadratic Lyapunov Function and Hybrid Control
} 

\title{\LARGE \bf
Safety Control Synthesis with Input Limits: a Hybrid Approach
}

\author{ Gray C. Thomas$^1$, Binghan He, and Luis Sentis 
\thanks{\fontsize{8}{8}\selectfont This work was supported by NASA Space Technology Research Fellowship NNX15AQ33H (G.C.T), and a Longhorn Innovation Fund For Technology grant (B.H.). $^{1}${\tt\small gray.c.thomas@utexas.edu}. Authors are with The Departments of Mechanical Engineering (G.C.T., B.H.) or Aerospace Engineering (L.S.), University of Texas at Austin,
        Austin, TX 78712-0292, USA.}%
}

\newcommand\copyrighttext{%
  \scriptsize 
  Accepted for publication in American Control Conference (ACC)
  \textcopyright 2018 IEEE. Personal use of this material is permitted. Permission from IEEE must be obtained for all other uses, in any current or future media, including reprinting/republishing this material for advertising or promotional purposes, creating new collective works, for resale or redistribution to servers or lists, or reuse of any copyrighted component of this work in other works.
  DOI: \href{https://ieeexplore.ieee.org/abstract/document/8431457}{10.23919/ACC.2018.8431457}
  }
\newcommand\copyrightnotice{%
\begin{tikzpicture}[remember picture,overlay]
\node[anchor=south,yshift=10pt] at (current page.south)
{\fbox{\parbox{\dimexpr\textwidth-\fboxsep-\fboxrule\relax}{\copyrighttext}}};
\end{tikzpicture}%
}

\begin{document}

\maketitle
\thispagestyle{empty}
\pagestyle{empty}
\copyrightnotice

\begin{abstract}
We introduce a hybrid (discrete--continuous) safety controller which enforces strict state and input constraints on a system---but only acts when necessary, preserving transparent operation of the original system within some safe region of the state space. We define this space using a Min-Quadratic Barrier function, which we construct along the equilibrium manifold using the Lyapunov functions which result from linear matrix inequality controller synthesis for locally valid uncertain linearizations. We also introduce the concept of a barrier pair, which makes it easy to extend the approach to include trajectory-based augmentations to the safe region, in the style of LQR-Trees. We demonstrate our controller and barrier pair synthesis method in simulation-based examples.
\end{abstract}


\section{Introduction}



Controllers which ensure safe operation of dynamic systems despite un-trusted inputs are widely appreciated for their straightforward safety verification. They find application where safety is critical, and also where input foresight is unavailable. These systems must first guarantee future satisfaction of both state and input constraints on the dynamic system---with input limits being a critical complicating factor. As was famously argued in the first Bode lecture, input limits on the fuel rod controller explain the signal behavior minutes before the Chernobyl reactor's nuclear melt-down \cite{Stein2003CS}. A natural secondary goal is to maximize the region of the state space that the controller certifies as safe to use.

Unlike the reference governor \cite{Bemporad1998TAC}, which enforces state constraints by way of constrained model predictive control \cite{MayneEA2000Automatica}, safety controllers for nonlinear systems use the sub-level sets of a scalar function---a Lyapunov function, or one of several relaxations---to encode information about the safe region boundary.

Several barrier Lyapunov nonlinear approaches start by building a Lyapunov function which is infinite within the unsafe region. Backstepping \cite{TeeGeTay2009Automatica} and adaptive control techniques \cite{LiuTong2016Automatica} can then guarantee safety, if not input limits. Less restrictive barrier functions and barrier certificates must decrease only at the boundary \cite{PrajnaJadbabaie2004HSCC}---the sub-level set of zero. Barrier functions have been extended to PDEs \cite{AhmadiValmorbidaPapachristodoulou2017SCL}, and dynamical segregation in arbitrary manifolds \cite{WisniewskiSloth2016TAC}. Barrier Lyapunov functions can be constructed from a Lyapunov function and a barrier function \cite{AmesEA2017TAC}, but finding such functions is the classical art of nonlinear control practitioners.

\begin{figure}
	\centering
	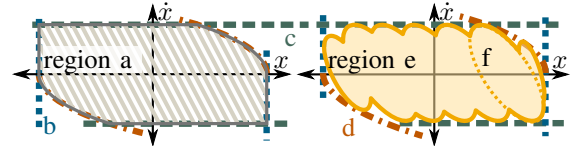
	\caption{The largest possible (\cite{BobrowDubowskyGibson1985IJRR}-style) safe region (a) for system $\ddot x=u$ with position limit (b lines), velocity limit (c lines) and acceleration or input limit (d lines). We approximate region (a) with (e), the zero sub-level set of our min-Quadratic barrier function, which is a union of robustly invariant ellipsoids (f) centered about various points on the equilibrium manifold.}\label{fig:spaces}
\end{figure}  

A control Lyapunov function \cite{Sontag1989SCL} merely needs to be capable of decreasing everywhere, and control barrier functions \cite{WielandAllgower2007IFAC} relax limitations on the choice of scalar function to the utmost---but still leaves the construction of such functions, and the bounding of the input, as an art. The secondary, or non-safety control objectives can be combined into a composite function \cite{RomdlonyBayu2016Automatica}, or added to the optimization which determines the input \cite{AmesEA2017TAC},  \cite{NguyenSreenath2016ACC}, \cite{NguyenSreenath2016IJRRreview}, but doing so alters the original controller's behavior everywhere. In particular, \cite{NguyenSreenath2016ACC} and \cite{NguyenSreenath2016IJRRreview} emphasize that high relative order constraints require careful adjustments to the boundary function to avoid large inputs.

Automatic synthesis of barrier certificates through sum of squares (SoS) optimization \cite{Prajna2006Automatica} has emerged as the standard solution to this design-burden issue, and has been adopted in safety verification \cite{BarryMajumdarTedrake2012ICRA}, and region of attraction estimation \cite{GlassmanEA2012ICRA} for already designed controllers in the presence of constraints. Most ambitiously, the LQR-Tree algorithm \cite{Tedrake2009RSS}, \cite{TedrakeEA2010IJRR} attempts to map out the entire backwards-reachable state space using the union of funnels---the sub-level sets of trajectory tracking LQR Lyapunov functions. The LQR-Tree strategy could potentially be adapted to safety control, but remains structurally plagued by the non-conservative polynomial approximation of the dynamics, inability to exploit choices available during controller design, and---despite efforts to improve the speed by sacrificing guarantees \cite{ReistTedrake2010ICRA}---dimensional explosion of the SoS sub-problem, trajectory optimization sub-problem, and the tree structure in the full state space 
\cite{TedrakeEA2010IJRR}.

Construction of a safe region can also be pursued through polyhedral sets \cite{Blanchini1999Automatica}. This approach offers a necessary and sufficient condition for safety of polytopic linear differential inclusion models, improving over the merely sufficient conditions which result from quadratic Lyapunov function synthesis. However polyhedral sets scale badly in high dimensions, and are difficult to incorporate into synthesis problems.

Linear matrix inequality (LMI) controller synthesis problems (c.f. \cite{BoydGhaouiFeronBalakrishnan1994SIAM}) offer a conservative way to certify invariant ellipsoids as safe---and design controllers to maximize their area. Such invariant ellipsoids have been applied to input and state constrained linear systems \cite{HuLin2003TAC}, and it has been shown that under these conditions the convex hull of the regions is also invariant---for ellipsoids sharing a center. The less-explored, non-convex min-quadratic function (mentioned in \cite{HuLin2006CDC} for same-center ellipsoids)---which bears similarity to the minimization over quadratics that runs once to select the starting funnel in \cite{ReistTedrake2010ICRA}---is more easily adapted to our purpose.

In this paper we propose to use linear differential inclusions to approximate a nonlinear system at a grid of equilibriums---each a conservative approximation of the nonlinear model within some region of validity. For each equilibrium we use an LMI to find the linear feedback and quadratic Lyapunov function such that the function's unity-sublevel set satisfies all state and input constraints while certifying the largest volume region. Our min-quadratic barrier function is the minimum over all of these quadratic Lyapunov functions---minus one so that the 0-sub-level set is an approximation of the safe region. This produces an approximate safe region which is the union of ellipsoids, region e in Fig.~\ref{fig:spaces}.

The ideal safety controller, in our view, would adhere exactly to Fig.~\ref{fig:spaces}'s region (a), applying either no input, or limit-saturated input as soon as the state hits the boundary of the safe region. This strategy relies on being able to compute this region, but this is only feasible in SISO systems of order less than 2, in which case the safety boundary can be found by a series of integrations \cite{BobrowDubowskyGibson1985IJRR,ThomasSentis2016IROS,PhamEA2017IJRR}. In that spirit we apply a control guaranteed to reduce the min-quadratic barrier function only near zero, ensuring that trajectories which stay within a level set are unaffected by the safety control.

We demonstrate our technique by constructing the region in which an inverted pendulum can balance, subject to speed and input limits---demonstrating the natural emergence of an exponential deceleration limit near the point where the force of gravity overwhelms the pendulum. We also simulate the high relative order behavior of a series elastic actuator under position and motor effort constraints.

\section{Barrier Pairs}
\subsection{Problem Statement}
We consider the problem of designing a safety controller $\mathbf K$ and safe region $\mathcal X_0$ for the system $\Sigma_0$:
\begin{equation}
\dot x = f(x) + g(x) u
\end{equation}
which forms a safe closed loop system $\Sigma_s$, guaranteed to satisfy constraints $x\in \mathcal X$ and $u\in\mathcal U$ indefinitely, for all initial states in $\mathcal X_0\subseteq \mathcal X$.

\subsection{Barrier Pairs}
Since we have input constraints, we define a concept to stand in for the standard notion of barrier functions.
\begin{definition}
A \emph{Barrier Pair} is a pair of functions $(B,\ k)$ with two properties: invariance and constraint satisfaction;
\begin{gather}
-1<B(x)\leq 0, u=k(x) \implies \mathring B(x) < 0 , \label{eq:bp1}\\
B(x)\leq 0 \implies x \in \mathcal X,\ k(x) \in \mathcal U.\label{eq:bp2}
\end{gather}
Using notation $\mathring B$ to mean the (minimally-restrictive) upper right-hand Dini derivative with respect to time.
\end{definition}

Note that apart from the generalization of the derivative, this definition is more stringent than a barrier function, since not only $B$ but also $B-\epsilon\ \ \forall\ 0\geq\epsilon>-1$ must be a barrier function for the system $\dot x = f(x)+g(x)k(x)$. Additionally, while barrier functions can be expected to hold for a saturated-input system, they do not themselves uphold limits on the inputs. These requirements on $B$ also impose a constraint which is not present in (nor particularly convenient to include in) the definition of control barrier functions, that 
\begin{equation}
\underset{u\in\mathcal U}{\min}\mathring B(x)<0 \quad \forall\ x\in\{x: -1<B(x)\leq0\}.
\end{equation}

Barrier pairs permit a (discontinuous) version of \cite{WielandAllgower2007IFAC}'s Theorem 7 with a claim on the input bound:
\begin{prop}
Given system $\Sigma_0$, a barrier pair $(B,k)$, and a partially known policy, for any $0\geq\epsilon >-1$
\begin{equation}
k_0(x) = \left\lbrace \begin{aligned} k(x) && B(x)=\epsilon\\u\in\mathcal U &&\text{otherwise}\end{aligned}\right.
\end{equation}
then if $\mathbf K$ implements $u=k_0(x)$, $\Sigma_s$ is safe and $\mathcal X_0$ can be chosen
\begin{equation}
\mathcal X_0=\{x \ : \ B(x)\leq \epsilon\}
\end{equation}
\end{prop}

\begin{proof}
$B(x)=\epsilon \implies u=k(x) \implies \mathring B(x) < 0,\ \therefore \mathcal X_0$ is invariant, $x\in \mathcal X_0\implies B(x)\leq \epsilon \implies x \in \mathcal X,\ k_0(x) \in \mathcal U$.
\end{proof}

The lower bound on $B(x)$ in the invariance condition \eqref{eq:bp1} allows for Lyapunov functions, which do not decrease at the origin---or which, in some uncertain systems, do not decrease inside the boundary of a residual set---to be used in a barrier pair (after some minor shifting and scaling).

\subsection{Barrier Pair LMI Subproblem}
Consider a linear differential inclusion (LDI) model \cite{BoydGhaouiFeronBalakrishnan1994SIAM} which approximates $\Sigma_0$ near an equilibrium\footnote{Generating robust approximations does not have an easy catch-all solution, but is often possible with a little insight into the problem structure. While simply adding model conservatism is another option, we advise a rigorous empirical validation if the safety guarantees are important.}---a robust linearization. We will focus on polytopic LDIs\footnote{This is for simplicity; norm-bounded LDIs have the best scaling for high dimensional problems, can be directly identified \cite{ThomasSentis2018arXiv}, and are also robustly stable if and only if there exists a quadratic Lyapunov function. This is merely a sufficient condition for polytopic LDIs.} 
{\small
\begin{align}
\dot x \in \mathbf{Co}&\{A_l (x-x_e)+B_l(u-u_e),\   l=1,\dotsc,L\},\label{eq:RoV}\\
\forall \quad& x\in\{x\ : \  |a_i^T (x-x_e)|\leq \alpha_i, \ i=1,\dotsc,n_a \}\subseteq \mathcal X,\nonumber\\
&u\in \{u :\ |b_i^T (u-u_e)|\leq \beta_i, \ i=1,\dotsc,n_b\} \subseteq \mathcal U,\nonumber
\end{align}}
operator $\mathbf{Co}$ denoting convex hull,
which approximate $\Sigma_0$ near an equilibrium $(x_e,\ u_e)\ :\ f(x_e)+g(x_e)u_e=0$. 

This allows a fairly standard set of LMI constraints to determine a positive definite matrix $Q$, and full state feedback matrix $K$ such that---defining scalar $B(x)\triangleq(x-x_e)^T Q^{-1} (x-x_e)-1$ and ellipsoidal region $\mathcal E \triangleq \{x\ |\ B(x)\leq 0\}$---when $u = k(x) \triangleq u_e+K(x-x_e)$ and $x\in \mathcal E$, then $B(x)+1$ is a quadratic Lyapunov function, $x$ satisfies all state constraints, and $u$ satisfies all input constraints. 

Following the standard trick \cite{BoydGhaouiFeronBalakrishnan1994SIAM} to make this type of problem convex, we define $Y\triangleq KQ$ as an optimization variable---and extract $K$ from $Y$ and $Q$ after the problem is solved.

State constraints from \eqref{eq:RoV} are enforced such that $x\in\mathcal E\implies |a_i^T(x-x_e)|<\alpha_i$---a linear constraint on $Q$
\begin{equation}
a_i^T Q a_i \leq \alpha_i^2.\label{eq:state_lim}
\end{equation}
And input constraints in the form
\begin{equation}
\begin{pmatrix}
Q & Y^T b_i\\
b_i^T Y & \beta_i^2
\end{pmatrix}\succeq 0.\label{eq:input_lim}
\end{equation}
are added to ensure $x\in\mathcal E \implies |b^T(k(x)-u_e)|\leq \beta_i$.
To guarantee $B(x)<0,\ u=k(x) \implies \dot B(x)<0$ is equivalent to the standard Lyapunov condition,
\begin{equation} \label{constraint1-2-1}
A_lQ + QA_l^{T} + B_lY + Y^{T}B_l^{T} \prec 0 \ \forall \ l = 1, \dotsc, L.
\end{equation}
And finally, to maximize the volume of the ellipsoid $\mathcal E$, we maximize the log of the determinant of $Q$---a concave cost function \cite{BoydGhaouiFeronBalakrishnan1994SIAM}. With a numerical tolerance $\varepsilon >0$, and a minimum exponential decay rate $\lambda>0$, our sub-problem is to
{
\begin{equation} \label{optimization 2}
\begin{aligned}
& \underset{Q, \, Y}{\text{maximize}}
& & \mathrm{log}(\mathrm{det}(Q)) \\
& \text{subject to} & & Q\succeq \varepsilon I\\
&&& \eqref{eq:state_lim}\ \forall\  i = 1, \dotsc, n_a  \\
&&& \eqref{eq:input_lim}\ \forall\  i = 1, \dotsc, n_b  \\
&&& A_l Q + Q A_l^{T} + B_l Y + Y^{T} B_l^{T} + \varepsilon I+\lambda Q \preceq 0 \ \\
&&& \qquad \qquad \qquad\qquad\forall \ l = 1, \dotsc, L
\end{aligned}\nonumber
\end{equation}}

Which naturally provides a barrier pair $(B, k)$ if the problem is feasible.

\subsection{Combining Barrier Pairs}
\begin{prop}
For any list of barrier pairs $(B_1,k_1),\ (B_2, k_2),\  \dotsc,\ (B_N,k_N)$, the pair comprising the min-barrier function 
\begin{equation}
\mathbf B(x) \triangleq \underset{n=1,\dotsc,N}{\min} B_n(x) \label{eq:B}
\end{equation}
and control input\footnote{This input is occasionally ambiguous. The choice does not matter.}
\begin{equation}
\mathbf k(x) \triangleq k_n(x) \ |\ n\in \underset{n=1,\dotsc,N}{\argmin} B_n(x),\label{eq:k}
\end{equation}
$(\mathbf B, \mathbf k)$, is also a barrier pair.
\end{prop}
\begin{proof}
Consider the set
\begin{gather}
\mathcal N = \underset{n=1,\dotsc,N}{\argmin} B_n(x),
\shortintertext{and in particular an $n\in\mathcal N:\mathbf k(x)=k_n(x)$. Assuming first that $-1<\mathbf B(x)\leq 0$ and $u=\mathbf k(x)$,}
\mathring{\mathbf B}(x) \leq \mathring B_n(x)<0,
\shortintertext{since $(B_n, k_n)$ is a barrier pair, $u=\mathbf k(x)=k_n(x)$, and $-1<\mathbf B(x)=B_n(x)\leq 0$. This demonstrates \eqref{eq:bp1}. As for \eqref{eq:bp2}, using the same choice of $n$,}
0\geq \mathbf B(x)=B_n(x)\implies x\in \mathcal X,\ \mathbf k(x)=k_n(x)\in\mathcal U
\end{gather}
\end{proof}
We note that this combination technique is very similar to the initial funnel lookup operation in the LQR-Trees algorithm. Moreover, if they provide certain guarantees, funnels imply the existence of a barrier pair.
\begin{prop}

A funnel comprising a trajectory-centered Lyapunov function $V(t,x)$ ($=V_t(x)$) and control $k(t,x)$ ($=k_t(x)$) for all $t\in[0,\infty]$ (note the extended reals), such that the funnel volume---the union of unity sub-level sets of $V_t\ \forall\ t\in[0,\infty]$---satisfies all state constraints ($V_t(x)<1\implies x\in \mathcal X$) and input constraints ($V_t(x)<1\implies k_t(x)\in\mathcal U$) guarantees the existence of a barrier pair $(\mathbf B, \mathbf k)$ formed according to the following continuous parameter versions of \eqref{eq:B} and \eqref{eq:k}:
\begin{align}
\mathbf B(x) &= \underset{t\in[0,\infty]}{\min}V_t(x)-1,\\
\mathbf k(x) &= k_t(x) \ |\ t\in\underset{t\in[0,\infty]}{\argmin} V_t(x). 
\end{align}
This pair is the combination of an infinite list of pairs $(V_t-1,k_t)$, parameterized by a time parameter, each of which upholds constraint satisfaction \eqref{eq:bp2}, but does not individually have invariance \eqref{eq:bp1}---unless they happen to be equilibrium-centered. 
\begin{figure}
\centering
\scalebox{.9}{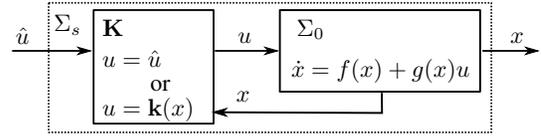}
\caption{Block diagram, re-purposed from \cite{WielandAllgower2007IFAC}, showing a safety controller $\mathbf K$ in feedback with the original system $\Sigma_0$ to produce a safe system $\Sigma_s$. The safety controller chooses either to be completely transparent ($u=\hat u$) or apply the known-to-be-safe input $u=\mathbf k(x)$.}\label{fig:safety_controller}
\end{figure}
\begin{figure}
\centering
\scalebox{.8}{
	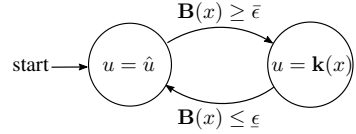}
    \caption{Hybrid Control System}
    \label{fig:fsm}
\end{figure}
\end{prop}
\begin{proof} For any x inside the funnel such that $\exists \  t\in[0,\infty] \ | \ 0<V_t(x)\leq 1$, the funnel Lyapunov function satisfies $\dot V<0$. Consider the instantaneous $t\in\underset{t\in[0,\infty]}{\argmin}V_t(x)$, $\mathbf B(x) = V_t(x)-1$, $\mathbf k(x) = k_t(x)$. For an infinitesimal time $\delta>0$,
\begin{gather}
\mathbf B(x(t+\delta)) =\underset{\tau\in[0,\infty]}{\min} V_{\tau}(x(t+\delta))-1\\
\leq V_{t+\delta}(x(t+\delta))-1= \mathbf B(x)+\delta \dot V(t)\\
\mathring{\mathbf B}(x)= \underset{\delta \rightarrow 0+}{\lim\sup} \frac{\mathbf B(x(t+\delta))-\mathbf B(x)}{\delta}\leq \dot V(t)<0
\end{gather}
Indicating that $(\mathbf B,\mathbf k$) satisfies \eqref{eq:bp1}.
\end{proof}
Note that, especially when the invariant funnel is large or the original trajectory is near to intersecting itself, applying $\mathbf k(x)$ is an entirely different behavior compared to applying $k(t,x)$---the barrier pair discards the information from the funnel's time parameterization.

\section{Hybrid Safety Controller}\label{multiV}

Equipped with barrier pair $(\mathbf B,\mathbf k)$, with its potentially non-smooth $\mathbf k$, we opt for an explicitly discontinuous safety controller (Fig.~\ref{fig:safety_controller}) with a simple state machine (Fig.~\ref{fig:fsm}) to produce hysteretic behavior---reminiscent of \cite{PrajnaJadbabaie2004HSCC}'s second example of a safe hybrid system.

Behavior is tuned by two near-zero thresholds $\bar \epsilon$ and $\ubar \epsilon$, $0\geq\bar \epsilon >\ubar \epsilon>-1$. As $\ubar \epsilon \rightarrow \bar \epsilon$, the safety controller enforces the inequality constraint $\mathbf B\leq \bar \epsilon $, and as $\ubar \epsilon \rightarrow -1$ it returns the system to the nearest equilibrium after each run-in with the safety limits. Detuning $\bar \epsilon$ from the ideal of zero can only reduce $\mathcal X_0$, but offers a hedge against real-world noisy signals in the computation of $\mathbf B(x)$. The gap between $\bar \epsilon$ and $\ubar \epsilon$ indirectly sets the rate of back and forth switching when the system is up against the limit.


In the examples, we use a min-quadratic barrier pair---simply the combination of those barrier pairs resulting from the LMI sub-problem. We therefore expect that applying the control $\mathbf k(x)$ guarantees exponential convergence to one of the equilibriums---though which one, and whether the system will transition between local control laws as it settles is not clear before hand. (This behavior is later visualized in Fig.~\ref{fig:manyequilibriums}.) The original equilibriums now represent the multiple minima of $\mathbf B$, all of them sharing the minimum value $-1$. 

Note that the model on which the LMI is based is trusted  even at high frequencies, and claims a decrease in $\mathbf B$ the instant the control law is suddenly enabled. Deviation from this ideal behavior must be compensated for when tuning the switching thresholds, or uncertain models which are trusted (to be sufficiently uncertain) at high frequencies must be used.

\section{Examples}

In this section, we provide two simulation examples demonstrating the operation of the hybrid safety controller. The first example is a second order unstable nonlinear system, an inverted pendulum\footnote{We imagine a pendulum for which falling is a catastrophic failure.} (Fig.~\ref{fig:inverted_pendulum}). The second example is a spring-mass system with 1 input and 4 states---which we use to explore the behavior near high relative-order constraints. In both examples, we use only equilibrium-centered barrier pairs---generated using our example LMI subproblem.

\subsection{Inverted Pendulum System}

\begin{figure}
\centering

\vspace{1.5em} 

\hbox{\hspace{-1.5em} \scalebox{1.0}{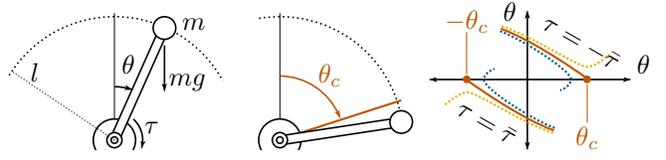}}
\caption{Inverted pendulum model with natural position limit $\theta_c$---a ``point of no return''---due to input limit $-\bar \tau \leq \tau \leq \bar \tau$, which causes $\theta_c$ to be a critical point in the flow field for $\tau = -\bar\tau$ and $-\theta_c$ to be one in the flow field for $\tau = \bar \tau$. In our approach such points are implicitly treated as on the boundary of the unsafe set.}
\label{fig:inverted_pendulum}
\end{figure}


\begin{figure*}
    \centering
    \scalebox{.85}{
    	\def\svgwidth{1.0\textwidth}
    	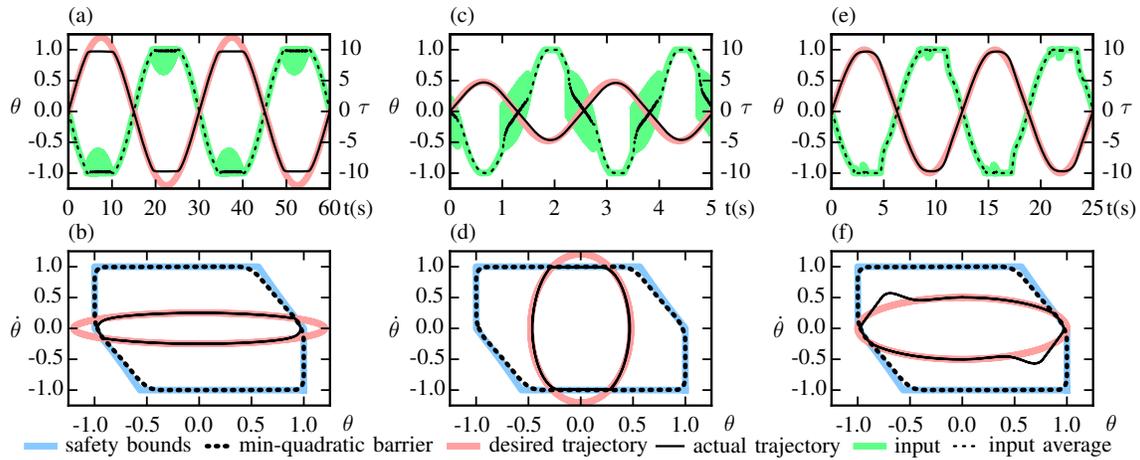
        }
    \caption{The inverted pendulum system protected by the safety controller, with a low priority reference tracking task for three references---one which exceeds position bounds (a,b), one which exceeds velocity bounds (c,d) and one which would just barely push this input limited unstable system past the point of no return (e,f). Each test shown both in the time domain (a,c,e), and in the phase space (b,d,f). Chattering input signal visualized raw and with low-pass filter. Input signal in low speed (a) and (e) examples dominated by gravity bias, hence behavior mostly opposite the position. 50 equilibrium point resolution. Signal tracking after leaving $\mathbf B(x)= \underline \epsilon$ implemented with a simple feedback linearizing controller.}
    \label{inverted pendulum plots}
\end{figure*}


We consider an inverted pendulum (Fig. \ref{fig:inverted_pendulum}, $m = 1 \si{kg}$, $l = 1.213 \si{m}$, $g=9.8\si{m/s/s}$), with safe region
\begin{align}
 \mathcal X &= \left\lbrace\begin{bmatrix} {\theta } \\ \dot{\theta } \end{bmatrix}\ : \ |\theta| \leq \theta_c = 1 \; \si{rad}, \ |\dot{\theta}| \leq 1 \; \si{rad /s} \right\rbrace ,\nonumber\\
\mathcal U &= \{\tau :\ |\tau| \leq \bar \tau = 10 \; \si{N \cdot m } \} ,
\end{align}
and dynamics $m l ^ {2} \ddot{\theta} = \tau + m g l \cdot \sin(\theta).$

Linearizing around equilibrium $\theta_e,\tau_e$, with a validity region\footnote{This validity region could potentially be iteratively tuned to match the extent of the ellipse, but such concerns are of lesser importance than verifying safety. Indeed, few systems will even have a readily available function for mapping these tuned regions to trustworthy models.} $|\theta-\theta_e|<\alpha = \min(0.25, \theta_c-|\theta_e|)$
\begin{small}
\begin{align*}
\begin{bmatrix} 
\dot {\theta }  \\
\ddot{\theta }    
\end{bmatrix} 
\in \mathbf{Co}&\left\lbrace
\begin{bmatrix} 
0 & 1 \\
\frac{g}{l} cos(\theta_{e}) \pm \bar\zeta & 0 
\end{bmatrix}
\begin{bmatrix} 
     \theta  - \theta_{e} \\
\dot{\theta }  - \dot{\theta }_{e}   
\end{bmatrix}
+
\begin{bmatrix} 
0 \\
\frac{1}{m l ^ {2}}
\end{bmatrix}
(\tau - \tau_{e})
\right\rbrace,\\ 
\forall \quad 
\begin{bmatrix} 
{\theta }  \\
\dot{\theta }    
\end{bmatrix}
&\in\left\lbrace
\begin{bmatrix} 
{\theta }  \\
\dot{\theta }    
\end{bmatrix}
\ : \ |\theta-\theta_e|\leq \alpha, \ |\dot{\theta}| \leq 1 \; \si{rad \cdot s ^ {-1}} \right\rbrace\subseteq \mathcal X,\nonumber \\
\tau&\in \mathcal U,
\end{align*}
\end{small}
where $\bar \zeta = \underset{\theta:|\theta-\theta_e|<\alpha}{\max}\zeta(\theta)$ represents a bound on linearization error,
\begin{equation}
\zeta(\theta) =  \frac{g}{l} \bigg[ \frac{sin (\theta) - sin (\theta_e)}{\theta - \theta_e} - cos(\theta_e) \bigg].
\end{equation}
While it is possible to analytically calculate this bound, it is also simple to compute via one dimensional brute force search.\footnote{This is not a conservative strategy in general, but it is an extremely accurate approximation relative to the numerical tolerances in the LMI subproblem.} 

This brings us to the LMI subproblems: we construct 50 barrier pairs to approximate the safe region using ellipsoids, and then combine them---forming a barrier pair with a  min-quadratic barrier function. Grid density is empirically tuned---checking for ellipsoid coverage does not scale well.

In the simulation (Fig.~\ref{inverted pendulum plots}), the inverted pendulum system is protected by the safety controller, with the safe system $\Sigma_s$ itself in a feedback configuration with a reference tracking controller. We demonstrate the behavior using three references. For a reference exceeding position bounds (Fig. \ref{inverted pendulum plots} a,b), the pendulum stops very close to the position bound and returns to tracking after the reference returns to $\mathcal X_0$. In the mean time, constraints are enforced by high speed switching. For a reference exceeding velocity bounds (Fig.~\ref{inverted pendulum plots} c,d), the pendulum stalls at the maximum allowable velocity. For a reference which would just barely push this input-limited unstable system past the point of no return (Fig.~\ref{inverted pendulum plots} e,f), the pendulum begins to rail the deceleration in advance of impact, and comes to a full stop in the safe region. When the reference returns to $\mathcal X_0$ it is moving relatively fast, and the reference-tracker has to exceed this speed to catch up. While this last-second deceleration behavior is not as perfect as is possible with second order systems, it is close---and this is encouraging for the higher order systems for which no equally simple policy exists.

The time plots (Fig. \ref{inverted pendulum plots} a,c,d) show the pendulum never violates any state or input limits during the tasks. By comparing the time plots between positions and inputs, the safety controller only applies $\mathbf k(x)$ when it is necessary. The chattering (fast switching) of the input happens because our $\bar \epsilon \approx \ubar \epsilon \approx 0$. When the pendulum is in the safety region, the performance of the reference tracking controller is preserved. 

\subsection{Double Spring-Mass}

\begin{figure}
\centering
\begin{tikzpicture}
\tikzstyle{spring}=[thick,decorate,decoration={zigzag,pre length=0.3cm,post length=0.3cm,segment length=6}]
\tikzstyle{damper}=[thick,decoration={markings,  
  mark connection node=dmp,
  mark=at position 0.5 with 
  {
    \node (dmp) [thick,inner sep=0pt,transform shape,rotate=-90,minimum width=15pt,minimum height=1pt,draw=none] {};
    \draw [thick] ($(dmp.north east)+(2pt,0)$) -- (dmp.south east) -- (dmp.south west) -- ($(dmp.north west)+(2pt,0)$);
    \draw [thick] ($(dmp.north)+(0,-5pt)$) -- ($(dmp.north)+(0,5pt)$);
  }
}, decorate]
\tikzstyle{ground}=[fill,pattern=north east lines,draw=none,minimum width=0.75cm,minimum height=0.15cm]

\begin{scope}[xshift=7cm]
\node [draw,outer sep=0pt,thick] (M1) [minimum width=1cm, minimum height=0.5cm] {$M_1$};

\node [draw,outer sep=0pt,thick] (M2) [minimum width=1cm, minimum height=0.5cm, xshift = 2.5cm] {$M_2$};

\node (ground1) [ground,anchor=north, yshift=-0.25cm, minimum width=1.5cm] at (M1.south) {};

\draw (ground1.north east) -- (ground1.north west);

\node (ground2) [ground,anchor=north, yshift=-0.25cm, minimum width=1.5cm] at (M2.south) {};

\draw (ground2.north east) -- (ground2.north west);

\draw [thick] (M1.south west) ++ (0.2cm,-0.125cm) circle (0.125cm)  (M1.south east) ++ (-0.2cm,-0.125cm) circle (0.125cm);

\draw [thick] (M2.south west) ++ (0.2cm,-0.125cm) circle (0.125cm)  (M2.south east) ++ (-0.2cm,-0.125cm) circle (0.125cm);

\node (wall) [ground, rotate=-90, minimum width = 1cm, yshift = - 2cm] {};

\draw (wall.north east) -- (wall.north west);


\draw [spring] ($(M1.north east)!(wall.90)!(M1.south east)$) -- ($(M2.north west)!(wall.90)!(M2.south west)$) node [midway,above] {$K$};


\draw [-latex,ultra thick] ($(M1.north east)$) ++ (0,0) -- +(.4cm,0) node [midway,above] {$u$};


 \draw[thick, dashed] ($(M1.north west)$) -- ($(M1.north west) + (0,.5)$);
 \draw[thick, dashed] ($(M2.north west)$) -- ($(M2.north west) + (0,.5)$);
 \draw[thick, -latex] ($(M2.north west) + (0,0.25)$) -- ($(M2.north west) + (0.5,0.25)$) node [midway, above] {$y_2$};
 \draw[thick, -latex] ($(M1.north west) + (0,0.25)$) -- ($(M1.north west) + (0.5,0.25)$) node [midway, above] {$y_1$};

\end{scope}
\end{tikzpicture}

\caption{A conceptual series elastic actuator model.} \label{double mass-spring}
\end{figure}
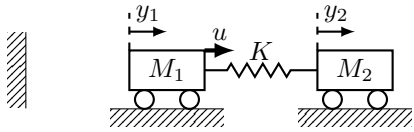

\begin{figure}
    \centering
    \scalebox{.85}{
    	\def\svgwidth{.5\textwidth}
    	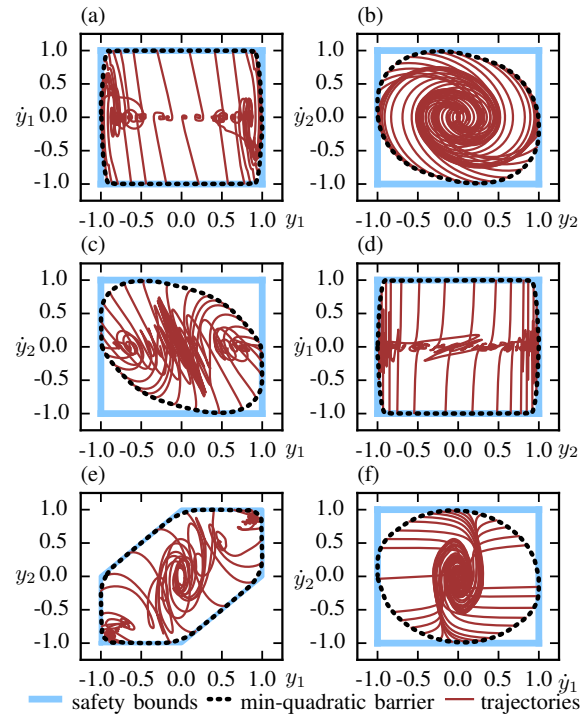
        }
    \caption{Visualizing the multiple equilibriums of $\mathbf k$ in the series elastic actuator example. Simulated trajectories begin on the extreme-projected edge of $\mathcal X_0$ for each plot, and are computed separately to demonstrate the invariance of $\mathbf B(x)<0$ when $u=\mathbf k(x)$.}
    \label{fig:manyequilibriums}
\end{figure}


A series elastic actuator can be conceptually modeled as a dual spring-mass system (Fig \ref{double mass-spring}) with $M_1$ as the motor, $M_2$ as the output inertia, and $u$ as the motor effort ($M_1=M_2 =1$, $K=1$). The safe constraints includes position limits, velocity limits, motor effort limits, and spring deflection limits:
\begin{align}
& \mathcal X = \{({y}_{1}, \, \dot{y}_{1}, \, {y}_{2}, \, \dot{y}_{2} ) ^ {T} : && \ |y_1 - y_2| \leq 1, \\
& && |y_i| \leq 1, \ |\dot y_i| \leq 1, \ i=1,2 \} ,\nonumber\\
& \mathcal U = \{u :\ |u| \leq 10 \} .\nonumber
\end{align} 

A linear state equation 
\begin{align}
M_1 \ddot y_1 &= K(y_2-y_1) + u\\
M_2 \ddot y_2 &= K(y_1-y_2)
\end{align}

is valid in any point of the safe region. We use 50 barrier pairs to approximate the safe state space region. 

30 trajectories are simulated on each of the 6 2D projections (Fig. \ref{fig:manyequilibriums}) of the state space. These trajectories start from the edge of the min-quadratic barrier and converge to one of its 50 minima, and these projections offer a glimpse into the approximation performance of our strategy in this four dimensional space---which notably allows very tight adherence to the position limits in input, output, and spring deflection states.

\section{Discussion}
This paper presents a synthesis method for controllers that guarantee future satisfaction of all state constraints, subject to input-limited dynamics, by taking advantage of the guarantees available through LMI-based controller synthesis using trusted LDI models. We introduced the concept of barrier pairs---which make it easier to reason about the satisfaction of input limits---and a min-quadratic barrier in particular, as a simple means of combining the results of many LMI-synthesis problems. And we distinguish our work by addressing input constraints, and by synthesizing the barrier function and the controller together---so that the controller choices, which can significantly alter the shape of $B(x)$, are used to maximize its volume. In the presence of input constraints, this controller synthesis sub-problem will choose the invariant ellipse to avoid the critical points and critical-point-terminating trajectories which represent dynamic limitations due to input limits.




\bibliographystyle{IEEEtran}
\bibliography{main}

\begin{thebibliography}{10}
\providecommand{\url}[1]{#1}
\csname url@samestyle\endcsname
\providecommand{\newblock}{\relax}
\providecommand{\bibinfo}[2]{#2}
\providecommand{\BIBentrySTDinterwordspacing}{\spaceskip=0pt\relax}
\providecommand{\BIBentryALTinterwordstretchfactor}{4}
\providecommand{\BIBentryALTinterwordspacing}{\spaceskip=\fontdimen2\font plus
\BIBentryALTinterwordstretchfactor\fontdimen3\font minus
  \fontdimen4\font\relax}
\providecommand{\BIBforeignlanguage}[2]{{%
\expandafter\ifx\csname l@#1\endcsname\relax
\typeout{** WARNING: IEEEtran.bst: No hyphenation pattern has been}%
\typeout{** loaded for the language `#1'. Using the pattern for}%
\typeout{** the default language instead.}%
\else
\language=\csname l@#1\endcsname
\fi
#2}}
\providecommand{\BIBdecl}{\relax}
\BIBdecl

\bibitem{Stein2003CS}
G.~Stein, ``Respect the unstable,'' \emph{IEEE Control Systems}, vol.~23,
  no.~4, pp. 12--25, 2003.

\bibitem{Bemporad1998TAC}
A.~Bemporad, ``Reference governor for constrained nonlinear systems,''
  \emph{IEEE Transactions on Automatic Control}, vol.~43, no.~3, pp. 415--419,
  1998.

\bibitem{MayneEA2000Automatica}
D.~Q. Mayne, J.~B. Rawlings, C.~V. Rao, and P.~O. Scokaert, ``Constrained model
  predictive control: Stability and optimality,'' \emph{Automatica}, vol.~36,
  no.~6, pp. 789--814, 2000.

\bibitem{TeeGeTay2009Automatica}
K.~P. Tee, S.~S. Ge, and E.~H. Tay, ``Barrier {L}yapunov functions for the
  control of output-constrained nonlinear systems,'' \emph{Automatica},
  vol.~45, no.~4, pp. 918--927, 2009.

\bibitem{LiuTong2016Automatica}
Y.-J. Liu and S.~Tong, ``Barrier {L}yapunov functions-based adaptive control
  for a class of nonlinear pure-feedback systems with full state constraints,''
  \emph{Automatica}, vol.~64, pp. 70--75, 2016.

\bibitem{PrajnaJadbabaie2004HSCC}
S.~Prajna and A.~Jadbabaie, ``Safety verification of hybrid systems using
  barrier certificates,'' in \emph{Hybrid Systems: Computation and Control},
  ser. Lecture Notes in Computer Science, 2004, vol. 2993, pp. 477--492.

\bibitem{AhmadiValmorbidaPapachristodoulou2017SCL}
M.~Ahmadi, G.~Valmorbida, and A.~Papachristodoulou, ``Safety verification for
  distributed parameter systems using barrier functionals,'' \emph{Systems \&
  Control Letters}, vol. 108, pp. 33--39, 2017.

\bibitem{WisniewskiSloth2016TAC}
R.~Wisniewski and C.~Sloth, ``Converse barrier certificate theorems,''
  \emph{IEEE Transactions on Automatic Control}, vol.~61, no.~5, pp.
  1356--1361, 2016.

\bibitem{AmesEA2017TAC}
A.~D. Ames, X.~Xu, J.~W. Grizzle, and P.~Tabuada, ``Control barrier function
  based quadratic programs for safety critical systems,'' \emph{IEEE
  Transactions on Automatic Control}, vol.~62, no.~8, pp. 3861--3876, 2017.

\bibitem{BobrowDubowskyGibson1985IJRR}
J.~E. Bobrow, S.~Dubowsky, and J.~Gibson, ``Time-optimal control of robotic
  manipulators along specified paths,'' \emph{The International Journal of
  Robotics Research}, vol.~4, no.~3, pp. 3--17, 1985.

\bibitem{Sontag1989SCL}
E.~D. Sontag, ``A `universal' construction of {A}rtstein's theorem on nonlinear
  stabilization,'' \emph{Systems \& Control Letters}, vol.~13, no.~2, pp.
  117--123, 1989.

\bibitem{WielandAllgower2007IFAC}
P.~Wieland and F.~Allg{\"o}wer, ``Constructive safety using control barrier
  functions,'' \emph{IFAC Proceedings Volumes}, vol.~40, no.~12, pp. 462--467,
  2007.

\bibitem{RomdlonyBayu2016Automatica}
M.~Z. Romdlony and B.~Jayawardhana, ``Stabilization with guaranteed safety
  using control {L}yapunov--barrier function,'' \emph{Automatica}, vol.~66, pp.
  39--47, 2016.

\bibitem{NguyenSreenath2016ACC}
Q.~Nguyen and K.~Sreenath, ``Exponential control barrier functions for
  enforcing high relative-degree safety-critical constraints,'' in
  \emph{American Control Conference (ACC), 2016}.\hskip 1em plus 0.5em minus
  0.4em\relax IEEE, 2016, pp. 322--328.

\bibitem{NguyenSreenath2016IJRRreview}
------, ``Optimal robust safety-critical control for dynamic robotics,''
  \emph{International Journal of Robotics Research (IJRR), in review}, 2016.

\bibitem{Prajna2006Automatica}
S.~Prajna, ``Barrier certificates for nonlinear model validation,''
  \emph{Automatica}, vol.~42, no.~1, pp. 117--126, 2006.

\bibitem{BarryMajumdarTedrake2012ICRA}
A.~J. Barry, A.~Majumdar, and R.~Tedrake, ``Safety verification of reactive
  controllers for {UAV} flight in cluttered environments using barrier
  certificates,'' in \emph{Robotics and Automation (ICRA), 2012 IEEE
  International Conference on}.\hskip 1em plus 0.5em minus 0.4em\relax IEEE,
  2012, pp. 484--490.

\bibitem{GlassmanEA2012ICRA}
E.~Glassman, A.~L. Desbiens, M.~Tobenkin, M.~Cutkosky, and R.~Tedrake, ``Region
  of attraction estimation for a perching aircraft: A {L}yapunov method
  exploiting barrier certificates,'' in \emph{Robotics and Automation (ICRA),
  2012 IEEE International Conference on}.\hskip 1em plus 0.5em minus
  0.4em\relax IEEE, 2012, pp. 2235--2242.

\bibitem{Tedrake2009RSS}
R.~Tedrake, ``{LQR}-trees: Feedback motion planning on sparse randomized
  trees,'' in \emph{Robotics Science and Systems V}.\hskip 1em plus 0.5em minus
  0.4em\relax MIT Press, 2009.

\bibitem{TedrakeEA2010IJRR}
R.~Tedrake, I.~R. Manchester, M.~Tobenkin, and J.~W. Roberts, ``{LQR}-trees:
  Feedback motion planning via sums-of-squares verification,'' \emph{The
  International Journal of Robotics Research}, vol.~29, no.~8, pp. 1038--1052,
  2010.

\bibitem{ReistTedrake2010ICRA}
P.~Reist and R.~Tedrake, ``Simulation-based {LQR}-trees with input and state
  constraints,'' in \emph{Robotics and Automation (ICRA), 2010 IEEE
  International Conference on}.\hskip 1em plus 0.5em minus 0.4em\relax IEEE,
  2010, pp. 5504--5510.

\bibitem{Blanchini1999Automatica}
F.~Blanchini, ``Set invariance in control,'' \emph{Automatica}, vol.~35,
  no.~11, pp. 1747--1767, 1999.

\bibitem{BoydGhaouiFeronBalakrishnan1994SIAM}
S.~Boyd, L.~El~Ghaoui, E.~Feron, and V.~Balakrishnan, \emph{Linear Matrix
  Inequalities in System and Control Theory}.\hskip 1em plus 0.5em minus
  0.4em\relax SIAM, 1994.

\bibitem{HuLin2003TAC}
T.~Hu and Z.~Lin, ``Composite quadratic {L}yapunov functions for constrained
  control systems,'' \emph{IEEE Transactions on Automatic Control}, vol.~48,
  no.~3, pp. 440--450, 2003.

\bibitem{HuLin2006CDC}
T.~Hu, L.~Ma, and Z.~Lin, ``On several composite quadratic {L}yapunov functions
  for switched systems,'' in \emph{Decision and Control, 2006 45th IEEE
  Conference on}.\hskip 1em plus 0.5em minus 0.4em\relax IEEE, 2006, pp.
  113--118.

\bibitem{ThomasSentis2016IROS}
G.~C. Thomas and L.~Sentis, ``Towards computationally efficient planning of
  dynamic multi-contact locomotion,'' in \emph{Intelligent Robots and Systems
  (IROS), 2016 IEEE/RSJ International Conference on}.\hskip 1em plus 0.5em
  minus 0.4em\relax IEEE, 2016, pp. 3879--3886.

\bibitem{PhamEA2017IJRR}
Q.-C. Pham, S.~Caron, P.~Lertkultanon, and Y.~Nakamura, ``Admissible velocity
  propagation: Beyond quasi-static path planning for high-dimensional robots,''
  \emph{The International Journal of Robotics Research}, vol.~36, no.~1, pp.
  44--67, 2017.

\bibitem{ThomasSentis2018arXiv}
G.~C. {Thomas} and L.~{Sentis}, ``{Identifying H[infinity]-Models: An LMI
  Approach},'' \emph{ArXiv e-prints}, Feb. 2018, arXiv:1802.07695 [math.OC].

\end{thebibliography}
\end{document}